\documentclass[12pt]{amsart}
\usepackage{amssymb}
\usepackage{verbatim}
\usepackage{amscd}

\theoremstyle{plain}
\newtheorem{theorem}{Theorem}
\newtheorem{corollary}{Corollary}
\newtheorem{proposition}{Proposition}
\newtheorem{lemma}{Lemma}

\theoremstyle{definition}
\newtheorem{definition}{Definition}

\theoremstyle{remark}
\newtheorem{remark}{Remark}
\newtheorem{example}{Example}
\newtheorem*{notation}{Notation}

\newcommand{\Lat}{\operatorname{Lat}}
\newcommand{\Alg}{\operatorname{Alg}}

\newcommand{\ti}{\tilde}

\newcommand{\wh}{\widehat}

\newcommand{\ol}[1]{\overline{#1}}

\def\bbC{\mathbb C}
\def\bbD{\mathbb D}

\def\bbN{\mathbb N}

\def\bbT{\mathbb T}

\newcommand{\fA}{\mathfrak A}     \newcommand{\sA}{\mathcal A}
     \newcommand{\sB}{\mathcal B}

     \newcommand{\sF}{\mathcal F}

     \newcommand{\sI}{\mathcal I}
     
     \newcommand{\sJ}{\mathcal J}
     
     \newcommand{\sM}{\mathcal M}
     
\newcommand{\fO}{\mathfrak O}     \newcommand{\sO}{\mathcal O}

     \newcommand{\sS}{\mathcal S}

\newcommand{\al}{\alpha}
\newcommand{\be}{\beta}
\newcommand{\vpi}{\varphi}

\newcommand{\la}{\lambda}
\newcommand{\ep}{\epsilon}
\newcommand{\ga}{\gamma}
\newcommand{\de}{\delta}

\begin{document}

\title[Singly Generated]
{Singly Generated Radical Operator Algebras}

\author{Justin R. Peters}
\address{Department of Mathematics\\
    Iowa State University, Ames, Iowa, USA}

\email{peters@iastate.edu}



\renewcommand{\subjclass}{2020 Mathematics Subject Classification }

 \begin{abstract}
We examine two nonselfadjoint operator algebras: the weighted shift algebra, and the Volterra operator algebra. In both cases,
the operator algebra is the norm closure of the polynomials in the operator norm. In the case of the weighted shift algebra, the existence of a gauge action
allows us to apply Fourier analysis to study the ideals of the algebra. In the case of the Volterra operator algebra, there is no gauge action, and 
other methods are needed to study the norm structure and the ideals.
 \end{abstract}
\maketitle

Here we consider commutative operator algebras, which need not be self-adjoint. In the case of semi-simple commutative operator algebras, the Gelfand theory provides
a complete description. At the other extreme, there is at this point no comprehensive theory of commutative radical operator algebras. This paper deals primarily with two types of
commutative radical operator algebras: namely, those generated by weighted shifts, and the one generated by the Volterra integral operator.

Given a bounded linear operator $T$ on a complex Hilbert space $H,$ there are various topologies in which one take the closure of the polynomials in $T$
to form an operator algebra. In this paper, we deal with the operator-norm closure. 
Thus, by the operator algebra $\sA_T$  we mean the operator-norm closure of the linear subspace of $\sB(H)$
generated by $\{T, T^2, T^3, \dots \}.$

If $T$ is a bounded linear operator on the Hilbert space $H,$ then by definition the operator algebra $\sA_T$ is completely isometrically represented on $H,$
and C$^*(T),$ the C$^*$-algebra generated by $T$ in $\sB(H),$ is a C$^*$-cover. The coordinate-free study of $\sA_T$ would include, say, the determination
of the closed ideals of $\sA_T,$ rather than the invariant subspaces arising from the action of $\sA_T$ on the Hilbert space $H,$ or, say, the existence of certain automorphisms,
which is a property of the abstract operator algebra and not a particular representation. 
The coordinate-free study of operator algebras was stimulated by \cite{BRS1990}, which gave internal `matrix-norm' conditions for a
Banach algebra to be an operator algebra.
Our approach here is necessarily a hybrid, as most properties of the operator algebra
can only be deduced from the given representation, at least with the tools we have available.

One automorphism that has proved fruitful in the C$^*$-theory is the gauge automorphism. For example, this is useful in proving the simplicity of the Cuntz algebras
$\sO_n$ (e.g., \cite{KDav96}, Theorem V.4.6). However, gauge actions have been employed in nonselfadjoint operator algebras as well (\cite{HPP2005}). For an operator
$T \in \sB(H),$ we say that the operator algebra $\sA_T$ admits a gauge action if the map $T \mapsto zT \ (|z| = 1)$ extends to an isometric isomorphism of $\sA_T.$
(See Definition~\ref{d:gauge action}.)
The existence of a gauge action on $\sA_T$ allows for the application of Fourier analysis on the elements $S \in \sA_T,$ which in turn has application to the ideal structure of the algebra.
For some operators $T \in \sB(H)$ the associated operator algebra $\sA_T$ will admit a gauge action, while others will not. 
We show that if $T$ is a weighted shift operator, then $\sA_T$ admits
a gauge action. However if $V$ is the Volterra integral operator, then $\sA_V$ fails to admit a gauge action. This distinction implies that $\sA_T$ and $\sA_V$
are not isomorphic as operator algebras. (Corollary~\ref{c: not isomorphic})

Section 2 provides some background results regarding gauge actions and applications of gauge automorphisms to singly generated operator algebras $\sA_T,$ and
some basic examples of operators $T$ for which the associated algebra $\sA_T$ either does, or does not, admit a gauge action.

In Section 3 we consider operator algebras generated by weighted shift operators $T.$
 In addition to the operator norm on $\sA_T,$ there is a norm arising from a
cyclic and separating unit vector for $\sA_T.$ But in general, the Hilbert space norm and the operator norm are inequivalent. (Remark~\ref{r: inequivalent norms}) But if the weight
sequence is square summable, then the two norms are equivalent. (Proposition~\ref{p: equivalent norms}) Proposition~\ref{p: unit ball} gives a sufficient condition for an
element $S$ in the unit ball of $\sA_T$ to be an extreme point. In particular, the normalized powers of $T, \ T_n = \frac{1}{||T^n||}T^n$ are extreme points
of the unit ball of $\sA_T.$ Under the same conditions on the weights, there is an isomorphism of the lattice of closed ideals of $\sA_T$ and closed invariant subspaces.
(Proposition~\ref{p: ideals and subspaces}) We give two results describing which elements $S \in \sA_T$ generate gauge-invariant ideals. We conclude this section
showing that the operator algebra $\sA_T$ is a (nonunital) integral domain.

The final section of the paper deals with $\sA_V,$ the operator algebra generated by the classical Volterra integral operator on $L^2[0, 1].$ The  closure of
the polynomials in $V$ in the strong operator topology turns out to be the commutant of $V,$ and hence corresponds also to the weak and weak$^*$ closed algebras
generated by $V.$ (\cite{KDav1988} Theorem 5.10)  Another weakly closed algebra associated with $V$ is $\Alg(\Lat V),$ 
the weakly closed algebra of operators in $\sB(L^2[0, 1])$ which leaves the lattice 
of subspaces $\Lat V$ invariant. This algebra, which contains the commutant $\{V\}^c,$ is non-commutative (\cite{KDav1988} Theorem 5.12). The operator norm closed
algebra generated by $V, \ \sA_V$ by contrast is composed of operators which share important properties of $V$: any $S \in \sA_V$ is quasinilpotent and compact.
Furthermore, given $S \in \sA_V$ there is a measurable function $f$ on $[0, 1],$ integrable over compact subsets of $[0, 1),$ such that if $\rho \in L^2[0, 1],$
\[ S\rho(x) = \int_0^x f(x-t) \rho(t) \, dt  \text{ for almost all } x \in [0, 1]. (\text{Theorem }~\ref{t: Vf in L2}) \]
However, $f$ need not be integrable over $[0, 1],$ as shown in Example~\ref{e:notell1}. If $f$ is in $L^1[0, 1],$ then the $L^1$ norm of $f$ dominates the operator norm $||S||.$

While Theorem~\ref{t: Vf in L2} allows us to represent an arbitrary $S \in \sA_V$ as a operator defined by a kernel, it does not provide another tool to calculate or
estimate the norm. Even in the case of polynomials of low degree in $V$ little is known. Remarkably, the recent paper \cite{RW22} appears to 
be the first to have obtained an exact value for $||V^2||,$ expressed as the solution to a transcendental equation (\cite{RW22}, Corollary 3.2 \footnote{$||V^2|| = \eta_0^{-2},$
where $\eta_0$ is the least positive solution $\eta$ to the equation $\cosh(\eta)\cos(\eta) = -1.$ I recall many years ago
hearing that Paul Halmos had obtained an expression for $||V^2||$ as the solution of a transcendental equation, but cannot find a reference for it.}).
 Their computations are limited to polynomials of degree $2$ in $V.$  As a consequence of this lack of computational tools, we are not able to make any assertions
as to the extreme points of the unit ball of $\sA_V,$ as we did for the unit ball of the radical weighted shift algebra.

In \cite{PeWo99} it is shown that the nilpotent elements are dense in $\sA_V.$ We obtain the same result here as a consequence of Theorem~\ref{t: Vf in L2}.
The last result is an extension of Titschmarsh's theorem on zero divisors of $L^1[0, 1]$ to the operator algebra $\sA_V$ (Corollary~\ref{c: Titschmarsh}).

\section{Background, notation and examples}

\subsection{Gauge Actions and Fourier analysis on Singly Generated Algebras}

Let $T$ be a bounded linear operator on a complex Hilbert space $H,$ and $\sA_T$ the operator algebra in $\sB(H)$ which is the operator norm closure of
the polynomials in $T$ which vanish at the origin. Let $\bbT = \{ z \in \bbC: |z| = 1 \}.$

\begin{definition} \label{d:gauge action} Let $Aut(\sA_T)$ denote the group of isometric automorphisms of $\sA_T.$
We say that $\sA_T$ admits a \emph{gauge action} if there exists a continuous homomorphism $\ga: \bbT \to Aut(\sA_T)$ such that
$\ga_{\la}(T) = \la T \quad (\la \in \bbT).$ By a \emph{continuous homomorphism} we mean that for each $\la_0 \in \bbT$ and $S \in \sA_T,$
\[ ||\ga_{\la}(S) - \ga_{\la_0}(S)|| \to 0 \text{ as } \la \to \la_0  \text{ in } \bbT.\]
\end{definition}

\begin{remark} \label{r: isometric auto}
 In the examples of gauge actions that arise here, the gauge automorphisms are completely isometric.

\end{remark}

Assume that $\sA_T$ admits a gauge action. Then if $p$ is any polynomial with $p(0) = 0,\  \ga_{\la}(p(T)) = p(\la T).$  Since, by definition, the algebra $\sA_T$ is the 
norm closure of such polynomials in $T,$ it follows that the action of $\ga_{\la}$ on polynomials in $T$ determines the action of $\ga_{\la}$ on $\sA_T.$

Now if $p$ is a polynomial, $p(z) = \sum_{j=1}^n a_j z^j,$ then 
\[ \hat{p}(k)T^k = \int_{\bbT} \ga_\la(p(T))\, \la^{-k}\, d|\la| = \begin{cases} 
 a_k T^k \text{ if } 1 \leq k \leq n \\
 0 \text{ otherwise}
\end{cases} \]

Thus, for $S \in \sA_T,$
\[ \hat{S}(k)T^k = \int_{\bbT} \ga_{\la}(S)\, \, \la^{-k} \, d|\la| \]
is well-defined. We say that $\hat{S}(k) \in \bbC $ is the $k^{\text{th}}$ Fourier coefficient of $S.$

\begin{lemma} \label{l: unique}
If $\sA_T$ admits a gauge action then $S \in \sA_T$ is uniquely determined by its Fourier series.
\end{lemma}

\begin{proof}
It is enough to prove that if $S \in \sA_T$ is nonzero, then $\{\hat{S}(k)\}$ is not the zero sequence.

Suppose $S \neq 0,$ and that $\hat{S}(k) = 0$ for all $k.$  There is a continuous linear functional $\vpi$ on $\sA_T$ for which $\vpi(S) \neq 0,$ and hence the continuous function
$f(\la) = \vpi(\ga_{\la}(S))$ is nonzero. However,

\begin{align*}
\hat{f}(k) &= \int_{\bbT} \vpi(\ga_{\la}(S))\, \la^{-k} \, d|\la| \\
	&= \vpi(\int_{\bbT} \ga_{\la}(S) \, \la^{-k}\, d|\la|) \\
	&= \vpi(\hat{S}(k)T^k) \\
	&= \hat{S}(k) \vpi(T^k) \\
	&= 0
\end{align*}
This holds for $k = 1, 2 \dots,$ but also for $k \leq 0,$ since the Fourier coefficients $\hat{S}(k) = 0$ for all polynomials $S$ and hence for all $S \in \sA_T.$

This implies $f$ is identically zero, which is a contradiction.
\end{proof}

Just as with classical Fourier series, we associate with $S \in \sA_T$ the formal series 
\begin{equation} \label{eq: Fourier series}
 S \sim \sum_{j=1}^{\infty} \hat{S}(j) T^j
\end{equation}
We would like to construct a sequence of polynomials in $T$ which converges to $S$ in some sense.
To this end, let $\vpi $ be a continuous linear functional on $\sA_T$ and $p(z) = \sum_{j=1}^n a_j z^j$
a polynomial.  Then 
\[  \widehat{\vpi(p(T))}(k) = \int_{\bbT} \vpi(\ga_{\la}(p(T)))\, \la^{-k} \, d|\la| = a_k \vpi(T^k)  \]

Since an arbitrary $S \in \sA_T$ is a norm limit of polynomials, we have that
\[ \widehat{\vpi(S)}(k) := \int_{\bbT} \vpi(\ga_{\la}(S)) \, \la^{-k} \, d|\la| = a_k \vpi(T^k) \]
where $S \sim \sum_{j=1}^{\infty} a_j T^j.$

\begin{proposition} \label{p: reconstruction}
With notation as in the above paragraph, define the function $f: \bbT \to \bbC, \ f(\la) = \vpi(\ga_{\la}(S)).$  Then
\begin{enumerate}
\item[1] $f$ is a continuous function on $\bbT$ with $\hat{f}(n) = 0$ for $n \leq 0.$ 
\item[2] The sequence of functions 
\begin{align*}
s_n(\la) &= \sum_{j=1}^n \frac{n-j}{n} \hat{f}(j)\la^j \\
	&= \sum_{j=1}^n \frac{n-j}{n} a_j \vpi(T^j) \la^j \ (\la \in \bbT)
\end{align*}
converges uniformly in $\la$ to $f.$
\item[3] The sequence $\{ S_n(\la) = \sum_{j=1}^n \frac{n-j}{n} a_j T^j \la^j \}$ converges weakly to $\ga_{\la}(S),$ uniformly in $\la \in \bbT.$
\item[4] There is a sequence $R_n$ in the convex hull of the sequences $\{ S_n(1): n = 1, 2, \dots\}$ which converges in norm to $S.$
\end{enumerate}
\end{proposition}

\begin{proof}
1. By assumption, the map $\la \in \bbT \mapsto \ga_{\la}(S) \in \sA_T$ is norm continuous, and since $\vpi$ is norm continuous, it follows that $f: \bbT \to \bbC$ is
continuous. Now $\hat{S}(k) = 0 $ for $k \leq 0,$ so the same holds for $f.$

2. By Fejer's Theorem, the sequence of arithmetic means of the partial sums of the Fourier series for $f$ converges uniformly to $f$ on $\bbT.$

3. Since, for an arbitrary continuous linear functional $\vpi, \ \vpi(S_n(\la)) = s_n(\la),$ this is just a restatement of [2].

4. Follows from [3] by taking $\la = 1$ and applying the Hahn-Banach separation theorem.
\end{proof}

\begin{notation} \label{n: unitization}
The unitization of $\sA_T$ will be denoted $\ti{\sA}_T.$
\end{notation}
At times it will be convenient to work in $\ti{\sA}_T.$ The gauge action $\ga$ extends naturally to $\ti{\sA}_T$
with $\ga_{\la}(I) = 1 .$ Of course for $S \in \ti{\sA}_T, \ \hat{S}(0)$ may be nonzero.


\subsection{Nonselfadjoint operator algebras which admit a gauge action}

\begin{example} \label{e: nilpotent}
Let $\sM_2$ be the C$^*$-algebra of $2\times 2$ matrices, with standard matrix units $e_{i, j}, \  {1 \leq i, j \leq 2}.$ Let $T = e_{1,2}.$
Then $\sA_T$ admits a gauge action. Indeed, since $T^2 = 0,$ the operator space $\sA_T = \bbC \cdot T$
is one-dimensional, and the map $\ga_{\la}$ is a linear map with $\ga_{\la}(aT) = \la aT, \ a \in \bbC.$

To see that $\ga$ is completely isometric, it suffices to show that it extends to $\sM_2.$
Define $ U = e_{1, 1} + \la e_{2, 2}.$ Then $ \ga_{\la}(A) = U^* A U \  A \in \sM_2, \la \in \bbT$ extends the action of $\ga$ on $\sA_T$ to the
C$^*$-envelope, $\sM_2.$

Alternatively, we can invoke the description of $\sM_2$ as the universal C$^*$-algebra generated by an operator $T$ which is nilpotent of index $2$ satisfying
\[ T^*T + TT^* = I \]
$T \in \sB(H)$ for some Hilbert space $H,$ and $I$ the identity in $\sB(H).$  Since $\la T$ satisfies these same conditions for $ \la \in \bbT,$
it follows from the universal property that $T \mapsto \la T$ is automorphism of $\sM_2.$

\end{example}

\begin{example} \label{e: disc algebra}
Let $T$ be the multiplication operator on $L^2(\bbT), \ T\xi(z) = z\xi(z).$  The  unital algebra $\ti{\sA}_T$ is the disc algebra $\sA(\bbD),$ and the algebra $\sA_T$ is
the subalgebra of functions $f$ satisfying $f   \perp 1,$ where $1$ is the constant function in $L^2(\bbT).$

The gauge action $\ga_{\la}$ is given by $\ga_{\la}(T)\xi(z) =\la z \xi(z).$  Thus for $f \in \sA(\bbD), \ \ga_{\la}f(z) = f(\la z).$ 
This is isometric, even completely isometric. Indeed, the C$^*-$envelope of $\sA(\bbD)$ is $C(\bbT),$ and the gauge action on the disc algebra is the restriction of the
gauge action on $C(T), \ \ga_{\la}(f)(z) = f(\la z).$

The Fourier series (as defined in equation~\ref{eq: Fourier series} ) of $f \in \sA(\bbD)$ is the usual 
Fourier series of the function $f.$
\end{example}

\begin{example} \label{e: dirichlet}
Let $\{ S_1, \dots S_d\}$ be isometries which satisfy the Cuntz relation $\sum_{j=1}^d S_j S_j^* = I.$ Now if $i_1, \dots, i_n \in \{ 1, \dots, d\}$ and 
$\mu = (i_1, \dots, i_n)$ we write $S_{\mu} = S_{i_1} \dots S_{i_n}$ and $|\mu| = n.$ Let $\sA$ be the Dirichlet algebra generated by the ``monomials''
$S_{\mu} S_{\nu}^*$ with $|\mu| \geq |\nu|.$  Then $\sA$ is a nonself-adjoint subalgebra of the Cuntz algebra $\sO_d.$  Note that $\sA$ is invariant under the canonical
gauge action on $\sO_d.$  Thus, $\sA$ admits a gauge action. The gauge action on this subalgebra of $\fO_n$ was considered in \cite{HPP2005}.
\end{example}

\begin{example} \label{e: popescu}
Let $\{S_1, \dots, S_d\}$ be the isometries of Example~\ref{e: dirichlet}.  If $\sA$ is the nonself-adjoint algebra generated by $\{S_1, \dots, S_d\} \subset \fO_n$,
then $\sA$ admits a gauge action, since it is invariant under the canonical gauge action on $\sO_d.$  This algebra is known as Popescu's noncommutative disc algebra.
\end{example}

\begin{example} \label{e: Cuntz isometry} 
Let $\{S_1, \dots, S_d\}$ be as in Example~\ref{e: dirichlet}.  Here we assume that these operators are represented in some Hilbert space $\sB(H).$  Choose one of the
isometries, say $S_1, $ and let $\ga$ be the canonical gauge action on $\sO_d.$ Since $\ga_{\la}(S_1) = \la S_1,$ it follows that the subalgebra $\sA_{S_1}$ 
generated by $S_1$ of the
Cuntz algebra $\sO_d$ is invariant under $\ga.$  Hence the gauge action on the Cuntz algebra $\sO_d$ restricts to a gauge action on $\sA_{S_1}.$
\end{example}

\begin{example} \label{e: graph}
A variety of examples can be constructed as subalgebras of graph C$^*$-algebras which admit gauge actions.  In this context one can obtain examples
which are analogues of examples~\ref{e: dirichlet}, \ref{e: popescu} and \ref{e: Cuntz isometry}, and where the generating isometries are replaced by
Cuntz-Krieger partial isometries.
\end{example}

Let $\sA$ be an operator algebra, and $\fA = \text{C}_{env}^*(\sA) $ be its C$^*$-envelope.  Then $\sA^* $ is an operator algebra defined as a subalgebra of $\fA.$
\begin{proposition} \label{p: adjoint algebras}
If $\ga$ is a gauge action on the operator algebra $\sA,$ then the adjoint algebra $\sA^*$ admits a gauge action, also denoted by $\ga$ defined by
\[ \ga_{\la}(A^*) = (\ga_{\bar{\la}}(A))^* \ \la \in \bbT, \ A \in \sA \]
\end{proposition}

The proof is routime. 

\begin{proposition} \label{p: extension} Every completely isometric automorphism of a unital operator algebra $\sA$ lifts to a $*$-automorphism of the C$^*$-envelope
C$^*_{\text{env}}(\sA),$ which fixes $\sA$ as a set.
\end{proposition}

This is Proposition~10.1 of \cite{DaKa2011}. This tells us that if a unital operator algebra $\sA$ admits a gauge action $\ga,$ then each $\ga_{\la}$
extends to an automorphism, which we also denote by ${\ga}_{\la},$ of the C$^*$-envelope, but does not immediately imply that the map $\la \in \bbT \mapsto {\ga}_{\la}$
is continuous on the C$^*$-envelope.

\subsection{Examples of operators in Hilbert space which do not admit a gauge action}

\begin{example} \label{e: projection}
Let $ 0 \neq P$ be a projection in $\sB(H).$ As in Example~\ref{e: nilpotent} $\sA_P$ is one-dimensional, but in this case does not admit a gauge action.  
Indeed, since $P = P^2,$ if $\ga$ were a gauge action on $\sA_P$ we would have
\[ \la P = \ga_{\la}(P) = \ga_{\la}(P^2) = \ga_{\la}(P) \la_{\la}(P) = \la^2 P, \ \la \in \bbT \]
which is absurd.

\end{example}

\begin{example} \label{e: linear dependence}
More generally, suppose that $T \in \sB(H)$ is such that, for some $n > 1, 0 \neq T^n$ and the set $\{ T, T^2, \dots T^n\}$ is linearly dependent. Then $\sA_T$ does not admit
a gauge action.

Indeed, suppose to the contrary that $\sA_T$ admits a gauge action $\ga,$ and, choosing a dependence relation of minimal degree,
we can assume  that $a_1T + \cdots + a_mT^m = 0, \ m \leq n$ and $a_m \neq 0.$

Then 
\[0 = \int_{\bbT} \ga_{\la}(\sum_{k=1}^m (a_kT^k) \la^{-m}\, d|\la| = a_mT^m\]
Since $T^m \neq 0,$ it follows that $a_m = 0,$  a contradiction.

\end{example}

\begin{example} \label{e: nonisometric} 
Let $H$ be a Hilbert space with orthonormal basis $\{e_n\}_{n=1}^{\infty},$ and let $T \in \sB(H)$ be the operator defined by
$Te_n = \frac{1}{n} e_n, \ n \geq 1.$ We claim that the operator algebra $\sA_T$ does not admit a gauge action.

Consider the operator $T - T^2 \in \sA_T.$ This is a compact, self-adjoint operator in $\sB(H),$ so its norm is the maximum of the absolute values of the eigenvalues.
$||T - T^2|| = ||(T - T^2)e_2||_2 = \frac{1}{4}.$

Suppose that $\sA_T$ admits a gauge action $\ga.$ Then $\ga_{\la}(T - T^2) = \la T - \la^2 T^2,$ so for $\la = -1$ we obtain $-T - T^2.$
Computing $||-T - T^2||$ we have $||-T - T^2|| = ||(-T - T^2)e_1||_2 = 2.$

This is a contradiction, since by definition the gauge action is isometric on $\sA_T.$

\end{example}



\subsection{Gauge invariant Ideals in Operator algebras with gauge actions}

Let $\sA_T$ be the operator algebra generated by an operator $T \in \sB(H),$ and suppose $\sA_T$ admits a gauge action $\ga.$  A closed ideal $\sJ \subset \sA_T$ is 
\emph{gauge invariant} if, whenever $S \in \sJ, $ then $\ga_{\la}(S) \in \sJ \ (\la \in \bbT).$

\begin{proposition} \label{p: gauge invariant ideal}
Let $\sJ \neq (0)$ be a gauge invariant ideal in $\sA_T.$  Then there exists $n \in \bbN$ such that $\sJ= {<}T^n{>}.$  That is, $\sJ$ is the closed ideal in $\sA_T$ generated by $T^n.$
\end{proposition}

\begin{proof}
Let $n = \inf\{k \geq 1: \hat{S}(k) \neq 0 \text{ for some } S \in \sJ\}.$ Thus, there exists $S \in \sJ$ with $\int_{\bbT} \ga_{\la}(S)\la^{-n} \, d|\la| = a_n T^n \neq 0.$
Since $\sJ$ is closed and gauge invariant, $T^n \in \sJ.$ It follows that any $S \in \sA_T $ with Fourier series $S \sim \sum_{k=n}^{\infty} c_k T^k \in \sJ.$
Thus, ${<}T^n{>} \subset \sJ.$ That is, the closed ideal generated by $T^n$ is contained in $\sJ.$

On the other hand, let $S \in \sJ.$ Then, by definition of $n, \ S $ has Fourier series of the form $\sum_{k=n}^{\infty} c_k T^k,$ so that $\sJ \subset {<}T^n{>}.$

\end{proof} 

One ideal which is invariant under the gauge action is the Jacobson radical; indeed, it is invariant under all isometric automorphisms.
\begin{corollary} \label{c: radicals}
Let $T \in \sB(H)$ be an operator such that $\sA_T$ admits a gauge action. Then either $\sA_T$ is semi-simple, or $\sA_T$ is radical.
\end{corollary}

\begin{proof}
Let $\sJ \neq (0) $ denote the Jacobson radical of $\sA_T.$  Since the Jacobson radical is invariant under all isometric automorphisms, 
by Proposition~\ref{p: gauge invariant ideal} it follows that if the 
Jacobson radical is nonzero, there is an $n \in \bbN$ such that
$\sJ = {<}T^n{>}.$  But if $T^n$ is quasinilpotent, that is, has spectrum $\{0\},$ it follows from the Spectral Mapping Theorem that $T$ has spectrum $\{0\}.$
Hence, the ideal generated by $T,$ which is $\sA_T$ is in the Jacobson radical.
\end{proof}

\begin{example} \label{e: nontrivial radical}
Here we note that it can happen that if $T \in \sB(H)$ does not admit a gauge action, then we can have $(0) \neq Rad(\sA_T) \neq \sA_T.$

Let $H = H_1 \oplus H_2$ and $T = I_1 \oplus N,$ where $I_1$ is the identity on $H_1$ and $N \in \sB(H_2)$ is a nonzero nilpotent, with $N^2 = 0.$
 Let $p(z) = z - z^2.$
Then $p(T) = 0 \oplus N \in Rad(\sA_T),$ so that while $\sA_T$ is not a radical algebra, it has a non-trival Jacobson radical.
\end{example}

The disc algebra $\sA(\bbD)$ has a rich lattice of ideals. (\cite{KHof62}) Not unexpectedly, there are few gauge invariant ideals.
\begin{corollary} \label{c: disc algebra}
If $\sJ$ is a gauge invariant closed ideal of $\sA(\bbD),$ then (in the notation of Example~\ref{e: disc algebra}) $\sJ = {<}z^n{>}$ for some $n \in \bbN.$
\end{corollary}

\begin{proof}
The conclusion follows immediately from Proposition~\ref{p: gauge invariant ideal}.
\end{proof}

\section{Operator algebras generated by weighted shifts} \label{s: weighted shifts}

In this section, $T$ will denote a weighted shift operator.  Let  $\{e_n\}_{n\geq 0}$ be an orthonormal basis for the Hilbert space $H,$  with
$Te_n = a_n e_{n+1}, \ n \geq 0,$ and $a_n \neq 0$ for all $n.$  Since $T$ is bounded, the sequence $\{a_n\}$ is bounded, and $||T|| = \sup_n |a_n|.$

We begin by showing that the operator algebra admits a gauge action.
\begin{lemma} \label{l:gauge action for shift}
Let $T$ be as above. Then $\sA_T$ admits a gauge action.
\end{lemma}

\begin{proof}
With $\{ e_n\}_{n \geq 0}$ as above, define the unitary $W_{\la}, \ \la \in \bbT,$ by $W_{\la}e_n = \la^n e_n.$
 Now
\[ W_{\la}TW_{\la}^*e_n = W_{\la} T(\bar{\la^n}e_n) = \bar{\la^n} a_n W_{\la} e_{n+1} = \bar{\la^n} \la^{n+1} a_n e_{n+1} = \la Te_n \]
holds for any $n \geq 0,$ and since the $\{e_n\}_{n \geq 0}$ form a basis, we have $W_{\la}TW_{\la}^* = \la T.$  Thus  the map
$\la \in \bbT \mapsto W_{\la}TW_{\la}^* \in \sB(H)$ is continuous., and so $\la \mapsto (W_{\la}TW_{\la}^*)^n = W_{\la}T^nW_{\la}^*$ is continuous, and hence 
$\la \mapsto W_{\la}p(T)W_{\la}^*$ for any polynomial $p$ with $p(0) = 0.$  Now if $S \in \sA_T$ and $\ep > 0$ is given, there is a polynomial $p$ with
$||p(T) - S|| < \ep/3.  $  Now let $\la_0 \in \bbT$ and $\de > 0 $ be such that if $ |\la - \la_0| < \de,$ then $||W_{\la}p(T)W_{\la}^* - W_{\la_0}p(T)W_{\la_0}^*|| < \ep/3.$
Then
\begin{align*}
||W_{\la}SW_{\la}^* - W_{\la_0}SW_{\la_0}^* || &\leq ||W_{\la}(S - p(T))W_{\la}^* || + \\
	& ||W_{\la}p(T)W_{\la}^* - W_{\la_0}p(T)W_{\la_0}^*|| + || W_{\la_0}(p(T) - S)W_{\la_0}^*|| \\
	&< \ep/3 + \ep/3 + \ep/3
\end{align*}
Thus $\ga_{\la}(S) = W_{\la}SW_{\la}^*$ is a gauge action on $\sA_T.$

\end{proof}

\begin{remark} \label{r: Cstar cover}
We claim, furthermore, that the action is completely isometric.  Now the C$^*$-algebra generated by $T$ in $\sB(H), \ C^*(T),$ is a C$^*$-cover for $\sA_T,$ 
and the action of $\ga_{\la}$ on $\sA_T$ is the restriction to $\sA_T$ of the automorphism $ S \in \text{C}^*(T) \mapsto {\ga}_{\la}(S) := W_{\la}SW_{\la}^*.$

While it is clear that ${\ga}_{\la}$ is isometric on the C$^*$-cover, it is not obvious that the map $\la \in \bbT \mapsto {\ga}_{\la}$ is continuous, since
$\la \mapsto W_{\la}$ is not continuous.

It is more convenient to work with the unital algebras $\ti{\sA}_T,\ \ti{\sA}^*_T.$  The C$^*$ cover of $\ti{\sA_T} \subset \sB(H)$ is  the closure in $\sB(H)$ of the union
\[ \bigcup_{n=1}^{\infty} (\ti{\sA}_T^* \ti{\sA}_T)^n \subset \sB(H) \]

Now since the action is continuous on $\ti{\sA}_T$ and $\ti{\sA}_T^*,$ (Proposition~\ref{p: adjoint algebras})
 it is continuous on $(\ti{\sA}_T^* \ti{\sA}_T)^n.$  And since it is isometric, it is thus
continuous on the closure of the union.

Thus, the gauge action on $\sA_T$ is the restriction of a gauge action on a C$^*$-cover.
\end{remark}

\begin{lemma} \label{l: partial sums}
Let $T$ be as in Lemma~\ref{l:gauge action for shift}  If $S \in \sA_T,$ and 
\[ Se_0 = \sum_{n=1}^{\infty} c_n e_n,  \text{ then } c_n = \hat{S}(n) a_0 \dots a_{n-1} \]

\end{lemma}

\begin{proof}
Let $W_{\la} \ (\la \in \bbT)$ be the family of unitary operators from Lemma~\ref{l:gauge action for shift}, so that $W_{\la}e_n = \la^n e_n.$
Let $v$ be a linear combination of basis vectors. Since $\int_{\bbT} \la^{-n} W_{\la} v \, d|\la|$ is a multiple of $e_n,$ it follows that,  
\[\int_{\bbT} \la^{-n} W_{\la} v \, d|\la| = <v, e_n> e_n \]
 This holds for arbitrary vectors in $H.$

Suppose $Se_0 = \sum_{k=1}^{\infty} c_k e_k.$ Then
\begin{align*}
c_n e_n &= <Se_0, e_n>e_n \\
	&= \int_{\bbT} W_{\la}(Se_0)\, \la^{-n}\, d|\la| \\
	&= \int_{\bbT} W_{\la}SW_{\la}^*e_0 \, \la^{-n} \, d|\la| \\
	& = (\int_{\bbT} W_{\la}SW_{\la}^* \la^{-n} \, d|\la|)e_0 \\
	&= (\int_{\bbT} \ga_{\la}(S) \, \la^{-n} \, d|\la|)e_0 \\
	&= \hat{S}(n)T^ne_0
\end{align*}
where we have used that $W_{\la}^*e_0 = e_0.$  

Thus, $c_n e_n = \hat{S}(n)T^ne_0,$ so that $c_n = \hat{S}(n) (a_0 a_1 \cdots a_{n-1}).$ 

\end{proof}

\begin{remark} \label{r: cyclic and separating vector}
Let $T$ be as above, and $\ti{\sA}_T$ the unitization of $\sA_T.$  Then the vector $e_0$ is a cyclic and separating vector for $\ti{\sA}_T.$  That it is cyclic is clear, for if
$v$ is any finite linear combination of basis vectors $v = \sum_{n=0}^N c_n e_n,$ let $p$ be the polynomial $p(z) = \sum_{n=0}^N \frac{c_n}{a_0 \cdots a_{n-1}}z^n$
(where the empty product is defined to be $1$), then $p(T)e_0 = v.$  

That $e_0$ is separating is also straightforward.  First note that if $S \in \ti{\sA}_T,$ there is a sequence of polynomials $\{p_n\}$ with $\{p_n(T)\}$ converging to $S$
in the norm of $\ti{\sA}_T,$ so that by definition of the norm, $p_n(T)v \to Sv$ for every $v \in H$ and in particular for $v = e_0.$  By Proposition~\ref{p: reconstruction} these
polynomials can be taken to be convex combinations of Fejer polynomials, so that $\hat{p}_n(k) \to \hat{S}(k)$ for every $k = 0, 1, \dots.$  So if $S \neq 0,$ there is some $k$ with
$\hat{S}(k) \neq 0,$ and so 
$Se_0 = \sum_{j=0}^{\infty} \hat{S}(j)a_0 \cdots a_{j-1}e_j \neq 0.$

\end{remark}

It is natural to ask for a description of the extreme points of the unit ball of $\sA_T.$  While that seems out of reach in our context, a sufficient condition is at hand.
\begin{proposition} \label{p: unit ball}
Let the weighteds of $T$ satisfy $|a_1| \geq |a_2| \geq \cdots $ and let $\sB = \{ S \in \sA_T: ||S|| \leq 1\}$ be the closed unit ball.  Then
$S \in \sB$ is an extreme point of $\sB$ if $||S|| = ||Se_0||_2 = 1.$  In particular, the elements $T_n := \frac{1}{||T^n||} T^n $ are extreme points of $\sB.$
\end{proposition}

\begin{proof} As noted in Remark~\ref{r: cyclic and separating vector}, the map $S \in \sA_T \mapsto ||Se_0||_2 $ is a norm on $\sA_T,$ satisfying
$||Se_0||_2 \leq  ||S||,$ and thus the map $S \mapsto Se_0$ maps the unit ball of $\sA_T$ into the unit ball of the Hilbert space $H.$ Since every unit vector
in Hilbert space is an extreme point of the unit ball in $H,$ it follows that if $||S|| = ||Se_0||_2 = 1,$ then $Se_0$ is extreme in the unit ball of $H,$
and \textit{ a fortori} $S$ is extreme in the unit ball of $\sA_T.$ In particular, the condition on the weights $|a_1| \geq |a_2| \geq \cdots$ guarantees
that the monomials $T^n$ assume their norm at $e_0,$ hence the normalized monomials $T_n$ are extreme points of the unit ball of $\sA_T.$

\end{proof}

\begin{remark} \label{r: inequivalent norms}
If the weights satisfy $|a_0| \geq |a_1| \geq |a_2| \geq \cdots $ then, as noted in the proof of Proposition~\ref{p: unit ball}, the norms
$||T^n||, \  ||T^ne_0||_2$ coincide. However, it need not be the case that the two norms coincide, or even are equivalent, on the operator algebra $\sA_T.$

To see this, let the weights satisfy $a_0 = a_1 = \dots = 1,$ and take
 $p_n(T) = T + T^2 + \cdots T^n \ (n \in \bbN),$ and let $v_n = \frac{1}{\sqrt{n}}(e_0 + e_1 + \cdots e_{n-1}).$
One calculates that
\[ ||p_n(T) v_n||_2 = \frac{\sqrt{2n^2 + 1}}{\sqrt{3}} \text{ while } ||p_n(T)e_0||_2 = \sqrt{n}. \]
Thus,
\[ \frac{||p_n(T)||}{||p_n(T)e_0||_2}  \geq \frac{\sqrt{2}}{\sqrt{3}}\,\sqrt{n} \]
so the norms are inequivalent on $\sA_T.$

\end{remark}

While the operator norm is not in general equivalent to the norm $S \mapsto ||Se_0||_2$ on $\sA_T,$ under certain restrictions the two norms are equivalent.

\begin{proposition} \label{p: equivalent norms}
Let $\{a_n\}_{n\geq 0}$ be a  sequence satisfying $|a_0| \geq |a_1| \geq |a_2| \geq \cdots$ with $ \sum_{n=0}^{\infty} |a_n|^2 := M^2 < \infty,$ and $a_n \neq 0$ for all $n.$
Then the operator norm on $\sA_T$ is equivalent to the norm $S \mapsto ||Se_0||_2.$
\end{proposition}

\begin{proof}
Let $p(z) = \sum_{j=1}^r c_j z^j.$  Then
\begin{align*}
||p(T)e_k||_2^2 &= \sum_{j=1}^r |c_j|^2 \, |a_k a_{k+1} \cdots a_{k+j-1}|^2 \\
	&\leq \sum_{j=1}^r |c_j|^2 \, (\frac{|a_k|}{|a_0|})^2 \, |a_0 a_1 \cdots a_{j-1}|^2 \\
	&\leq (\frac{|a_k|}{|a_0|})^2 ||p(T)e_0||^2_2
\end{align*}

Now let $v$ be a unit vector which is a finite linear combination of basis vectors, so $v = \sum_{\ell = 0}^N \be_{\ell} e_{\ell}$ with $\sum_{\ell = 0}^N |\be_{\ell}|^2 = 1.$ Thus
\begin{align*}
||p(T)v||_2  &\leq \sum_{\ell = 0}^N |\be_{\ell}| ||p(T)e_{\ell}||_2 \\
	&\leq \sum_{\ell = 0}^N |\be_{\ell}| \frac{|a_{\ell}|}{|a_0|} ||p(T)e_0||_2 \\
	&\leq \frac{1}{|a_0|}||p(T)e_0||_2 (\sum_{\ell = 0}^N |\be_{\ell}| \, |a_{\ell}| ) \\
	&\leq \frac{1}{|a_0|}||p(T)e_0||_2 (\sum_{\ell = 0}^N |\be_{\ell}|^2)^{\frac{1}{2}})(\sum_{k= 0}^N |a_k|^2)^{\frac{1}{2}}  \\
	&\leq  \frac{1}{|a_0|}||p(T)e_0||_2 (\sum_{\ell = 0}^N |\be_{\ell}|^2)^{\frac{1}{2}})(\sum_{k= 0}^{\infty} |a_k|^2)^{\frac{1}{2}}  \\
	&\leq \frac{M}{|a_0|} ||p(T)e_0||_2
\end{align*}

Now since $||p(T)v||_2 \leq \frac{M}{|a_0|} ||p(T)e_0||_2$ for a dense set of unit vectors $v,$ it follows that $||p(T)|| \leq \frac{M}{|a_0|} ||p(T)e_0||_2.$
Finally, since  we can approximate an arbitrary $S \in \sA_T$ by polynomials in $T$, so we conclude that
$||S|| \leq \frac{M}{|a_0|} ||Se_0||_2$ for all $S \in \sA_T.$
\end{proof}

As a result of the equivalence of the two norms, several results follow immediately. 

\begin{corollary} \label{c: partial sums}
Let the weighted shift $T$ be as in Proposition~\ref{p: equivalent norms}.  Then for $S \in \sA_T,$ the partial sums of the Fourier series,
\[ \sum_{k=1}^n \hat{S}(k) T^k \]
converge in norm to $S.$
\end{corollary}

\begin{corollary} \label{c: converge}
Let the weighted shift $T$ be as in Proposition~\ref{p: equivalent norms}.  Then the sequence 
\[ p_n(T) := \sum_{k=1}^n c_k T^k \]
converges in norm to an element $S \in \sA_T$ if and only if 
\[ \sum_{k=1}^{\infty} |c_k|^2 \, |a_0 \cdots a_{k-1}|^2 < \infty \]
\end{corollary}

Proposition~\ref{p: equivalent norms} not only tells us that the operator norm is equivalent to a Hilbert space norm, but gives a mapping
\[ \sF: \sA_T \to H, \ S \mapsto Se_0  \] 
which maps $\sA_T$ onto the closed subspace $H_1 $ spanned by the basis vectors $e_n : \ n \geq 1.$ One can also define $\ti{\sF}: \ti{\sA}_T \to H $ by $S \mapsto Se_0.$
It is easy to see how to adapt Proposition~\ref{p: equivalent norms} to the unital algebra $\ti{\sA}_T.$
Note that the unital algebra $\ti{\sA}_T$ maps onto $H.$ Since $\sF$ is a Banach space isomorphism, it gives a one-to-one map of closed subspaces of $\sA_T$ to
closed subspaces of $H_1,$ and similarly $\ti{\sF}$ maps closed subspaces of $\ti{\sA}_T$ onto closed subspaces of $H.$ Furthermore
\begin{proposition} \label{p: ideals and subspaces}
Suppose the weights of $T$ satisfy the conditions of Proposition~\ref{p: equivalent norms}.
\begin{enumerate}
\item The map ${\sF}$ is an isomorphism of the lattice of closed ideals of ${\sA}_T$ onto the lattice of closed $T$-invariant subspaces of $H_1.$
\item The map $\ti{\sF}$ is an isomorphism of the lattice of closed ideals of $\ti{\sA}_T$ onto the lattice of closed $T$-invariant subspaces of $H.$
\end{enumerate}
\end{proposition}

\begin{proof}
We prove only the second statement. 

 Let us first observe that the map $\ti{\sF}$ is $T$-equivariant.  That is, $T\ti{\sF}(S) = \ti{\sF}(TS), \ S \in \ti{\sA}_T.$ Indeed, this follows
immediately from the definition of $\ti{\sF}.$ We claim that a closed subspace $\sI \subset \ti{\sA}_T$ is a closed ideal if and only if it is $T$-invariant. Clearly, if $\sI$ is
an ideal in $\ti{\sA}_T,$ then it is $T$-invariant.  On the other hand, if a closed subspace $\sI \subset \ti{\sA}_T$ is $T$-invariant, then it is invariant under
multiplication by any polynomial in $T.$  Let $S \in \sI$ and $R \in \ti{\sA}_T.$ If $\{p_n\}$ is a sequence of polynomials such that $\{p_n(T)\}$  converges in norm to $R,$
then $p_n(T)S$ converges in norm to $RS.$  Thus, $\sI$ is a closed ideal.

Now clearly the map $\ti{\sF}$ maps closed subspaces of $\ti{\sA}_T$ to closed subspaces of $H,$ and since $\ti{\sF}$ is $T$-equivariant, it is an isomorphism of the lattice of closed
$T$-invariant subpaces of $\ti{\sA}_T$ onto the lattice of closed $T$-invariant subpaces of $H.$ But as shown above, the closed  $T$-invariant subspaces of $\ti{\sA}_T$
are exactly the closed ideals.
\end{proof}

We know from Proposition~\ref{p: gauge invariant ideal} that the gauge invariant ideals in the algebra $\sA_T$ generated by a unilateral weighted shift $T$ are all of the form
$<T^k>.$ Given an element $S \in \sA_T,$ one can ask when the ideal $<S>$ is of the form $<T^k>$ for some $k \in \bbN.$ 

\begin{corollary} \label{c: more ideals}
Let the weighted shift $T$ be as in Proposition~\ref{p: equivalent norms}, and let $S \in \sA_T.$ Suppose $\hat{S}(j) = 0, \ j = 1, \dots, k-1$ and $\hat{S}(k) \neq 0$ for some $k > 1.$
Then the closed ideal ${<}S{>} = {<}T^k{>}$ if and only the closed subspace generated by the vectors $Se_0, \ TSe_0, \ T^2 Se_0, \dots $ contains the basis vector $e_k.$
\end{corollary}

\begin{proof}
The condition $\hat{S}(j) = 0$ for $ j = 1, \dots, k-1$ implies that ${<}S{>} \subset {<}T^k{>}.$  Indeed, there is a sequence of polynomials $\{p_n\} \subset {<}T^k{>}$ converging in
norm to $S,$ hence $S \in {<}T^k{>},$ and so the closed ideal ${<}S{>} \subset {<}T^k{>}.$  Applying the map $\sF : \sA_T \to H,$ it follows that $Se_0$ is contained in the closed
invariant subspace generated by the vector $e_k.$

By Proposition~\ref{p: ideals and subspaces}, in order for the two ideals to coincide, the corresponding subspaces under the map $\sF$ must coincide. Thus it is necesessary and 
sufficient that the closed subspace generated by the vectors $Se_0, \ TSe_0, \ T^2Se_0, \dots$ equal the closed subspace generated by $e_k, \ Te_k \ T^2e_k , \dots,$
which is the subspace generated by the basis vectors $e_k, \ e_{k+1}, \ e_{k+2}, \dots.$  Thus, if the closed subspace generated by the vectors $Se_0, \ TSe_0, \ T^2Se_0, \dots$
contains the vector $e_k,$ by invariance it contains the vectors $e_{k+1}, \ e_{k+2}, \ \dots,$ and hence the two closed subspaces coincide.
\end{proof}

The weights $\{a_n\}$ satisfying the conditions of Proposition~\ref{p: equivalent norms} satisfy $\lim_n a_n = 0, $ so that the wighted shift $T$ is quasinilpotent, and hence
the algebra $\sA_T$ is radical. The consequences of the Lemma mentioned so far did not make direct use of the fact that the elements of the algebra are all quasinilpotent.
The following Proposition gives a different sort of criterion as to when an element $S \in \sA_T$ generates an ideal of the form ${<}T^k{>}.$
Here we do not need to assume the equivalence of the operator norm to the norm $S \mapsto ||Se_0||_2,$ rather we need only assume that the weighted shift $T$ 
is quasinilpotent.

Recall (\cite{Rid70}) that a necessary and sufficient condition for a weighted shift operator to be quasinilpotent is that
$ \lim_n  \sup_k  {|a_{k+1} \cdots a_{k+n}|}^{\frac{1}{n}} = 0.$

\begin{proposition} \label{p: some ideals}
Let $\{a_n\}_{n \geq 0}$ be a sequence of nonzero weights such that the unilateral weighted shift operator $Te_n = a_n e_{n+1}$ is quasinilpotent, and hence the
algebra $\sA_T$ is radical. Let $S \in \sA_T$ be a nonzero element such that $\hat{S}(j) = 0, \  j < k$ and $\hat{S}(k) \neq 0.$ If in the unital algebra $\ti{\sA}_T, \ 
S$ factors as $S = T^kQ$ for some $Q \in \ti{\sA}_T,$ then ${<}S{>} = {<}T^k{>}.$ In particular, if $S$ is a polynomial $\hat{S}(k)T^k + \cdots \hat{S}(n)T^n,$ then
\mbox{${<}S{>} = {<}T^k{>}.$}
\end{proposition}

\begin{proof}
First observe that ${<}S> \subset {<}T^k{>}.$  By Proposition~\ref{p: reconstruction}
there is a sequence of polynomials $\{p_n(T)\} \subset {<}T^k{>}$ converging to ${<}S{>}.$  Thus $S$ belongs to the closed ideal ${<}T^k{>},$ and hence
 ${<}S{>} \subset {<}T^k{>}.$

Now we prove the reverse containment.
Since multiplying $S$ by a nonzero constant does not change the ideal ${<}S{>},$ we may assume that $\hat{S}(k) = 1,$ so that if $S = T^kQ,$ then 
$\hat{Q}(0) = 1.$ Writing $Q = I - R,$ we have that $I - R$ is invertible in $\ti{\sA}_T$ with inverse $ I + R + R^2 + \cdots.$ Indeed, since $R \in \sA_T$ which is radical,
given any $r > 0, \ ||T^n|| \leq r^n$ for $ n \geq N_r.$ Thus, $S(I-R)^{-1} = S + SR + SR^2 + \cdots.$  Note that while $S(I-R)^{-1}$ is a  product in the unital algebra, 
the sum $ S + SR + SR^2 + \cdots.$ is computed in $\sA_T,$ and equals $T^k.$ It follows that $T^k \in {<}S{>},$ and hence ${<}T^k{>} \subset {<}S{>}.$

Finally observe that if $S$ is a polynomial in $T,$ then the factorization $S = T^k Q$ is realizable in $\ti{\sA}_T.$
\end{proof}

In \cite{Sh74}, Theorem 3, Shields characterizes the commutant of a weighted shift in terms of formal power series. In particular, that implies that the
commutant is an integral domain. It follows that the smaller algebra $\sA_T$ is also an integral domain, though in our case a non-unital integral domain.
The existence of the gauge action on $\sA_T$ allows us to deduce the same result.

\begin{proposition} \label{p: no integral domain}
Let $T$ be a weighted shift with nonzero weight sequence as in Lemma~\ref{l:gauge action for shift}. Then $\sA_T$ is a non-unital integral domain.

In particular, if $T$ is a quasinilpotent weighted shift, then the nonzero elements of $\sA_T$ are quasinilpotent and not nilpotent.
\end{proposition}

\begin{proof}
Let $ R, S \in \sA_T,  $ be nonzero elements. From Lemma~\ref{l: unique} we know that the Fourier coefficients of $R$ are not all zero, and similarly for $S.$
By Proposition~\ref{p: reconstruction} there is a sequence of polynomials $\{p_n\}$ (resp., $\{q_n\}$) which are convex combinations of Fejer polynomials, so that
$\{p_n(T)\} $ converges to $R$ (resp., $\{q_n(T)\}$ converges to $S$). In particular, it follows from the Fejer property that $\wh{p_n}(k) \neq 0$ implies $\hat{R}(k) \neq 0$
(resp., $\wh{q_n}(k) \neq 0$ imples $\hat{S}(k) \neq 0).$ Furthermore, for all $k, \ \wh{p_n}(k) \to \hat{R}(k) $ (resp., $\wh{q_n}(k) \to \hat{S}(k)$) as $n \to \infty.$

Let 
\[ j_0 = \min\{j: \hat{R}(j) \neq 0\} \text{ and } k_0 = \min\{k: \hat{S}(k) \neq 0 \}.\]

Now $\{p_n(T) q_n(T)\}$ converges in norm to $RS,$ and so $\{\wh{p_n q_n}(\ell)\}$ converges to $\wh{RS}(\ell)$ for all $\ell \in \bbN.$
If $\ell_0 = \min\{ \ell: \wh{p_n q_n}(\ell) \neq 0 \text{ for } n \text{ sufficiently large}\}$
then $\ell_0 = j_0 + k_0$ and $\wh{p_n q_n}(\ell_0) = \wh{p_n}(j_0) \wh{q_n}(k_0) .$ Hence
\[ \wh{RS}(\ell_0) = \lim_n \wh{p_n}(j_0) \wh{q_n}(k_0) \neq 0\]
so that $RS \neq 0.$

For the second statement, if $T$ is quasinilpotent, then the elements of $\sA_T,$ are quasinilpotnt, and the nonzero elements are not nilpotent as $\sA_T$ is an integral domain.

\end{proof}


\section{The Volterra operator algebra}

Let $V$ be the Volterra operator on $L^2[0,1],$ given by $V\xi(x) = \int_0^x \xi(t)\, dt.$ Then we know (\cite{Lax}) that for $n \geq 0,$
\begin{equation} \label{eq: Vn}\
 V^{n+1}\xi(x) = \frac{1}{n!} \int_0^x  (x-t)^n \xi(t)\, dt 
\end{equation}

Let $f$ be any $L^2[0, 1]$ function and let $V_f$ denote the operator on $L^2[0, 1]$ given by
\[ V_f \xi(x) = \int_0^x f(x-t) \xi(t)\, dt \]
Observe this is bounded, since
\begin{align*} 
|V_f \xi(x)| &= |\int_0^x f(x-t) \xi(t)\, dt| \\
	&\leq \int_0^x |f(x-t)| \, |\xi(t)|\, dt \\
	&\leq [\int_0^x |f(x-t)|^2\, dt]^{\frac{1}{2}} [\int_0^x |\xi(t)^2|\, dt]^{\frac{1}{2}} \\
	&\leq ||f||_2 ||\xi||_2
\end{align*}

Now, since any $L^2[0, 1]$ function $f$ is the $L^2$ limit of a sequence of polynomials, $\{p_n\},$ we have that
 \[ |(V_f - V_{p_n})\xi(x)| \leq ||f- p _n||_2 ||\xi||_2 \]
so that 
\begin{equation} \label{e: ell2}
  ||V_f - V_{p_n}|| \to 0 
\end{equation}

Let $\sA_V$ denote the Volterra operator algebra, by which we mean the operator norm closure of the polynomials $p$ in $V$ with $p(0) = 0.$  We have just
shown that  $V_f \in \sA_V$  if  $f \in L^2[0, 1].$   We would like to characterize arbitrary $T \in \sA_V.$


\begin{theorem} \label{t: Vf in L2} 
\begin{enumerate}
\item Let $f \in L^1[0, 1].$ Then for $\rho \in L^2[0, 1],$ the function $(V_f)\rho(x) := \int_0^x f(x-t)\rho(t) \, dt \in L^2[0, 1],$ and
\[ ||(V_f)\rho||_2 \leq  ||f||_1 \, ||\rho||_2 \text{ and hence } ||V_f|| \leq  ||f||_1 \]
\item $ f \in L^1[0, 1] \mapsto ||V_f|| $ is a norm on $L^1[0, 1].$
\item  Let $T \in \sA_V$ and $\{f_n\}$ a sequence of functions in $L^1[0, 1]$ such that $V_{f_n} \to T$ in $\sA_V.$
If $ 0 < x_0 < 1,$ then the sequence $\{ \int_0^{x_0} |f_n(x)|\, dx \}$ is bounded.
\item Let $ 0 < x_0 < 1$ and let $S_{x_0} = \{ f: f \in L^1 [0, 1], \ f(x) = 0 \text{ for } x_0 < x \leq 1 \}.$
	Then on $S_{x_0}$ the operator norm $ f \mapsto ||V_f||$ and the $L^1$-norm $f \mapsto \int_0^1 |f|$ are equivalent.
\item  Let $T \in \sA_V.$ Then there is a measurable function $f$ on $[0, 1],$ integrable over compact subsets of $[0, 1),$ such that $T = V_f.$
\item If $f, g$ are measurable functions on $[0, 1]$ such that $V_f, V_g \in \sA_V$ and $ \al \in \bbC,$ then $V_{\al f} = \al V_f$ and $V_{f+g} = V_f + V_g.$

\end{enumerate}
\end{theorem}

\begin{proof}

Let $\rho \in L^2[0, 1]$ and $f \in L^1[0, 1].$ For $x \in [0, 1]$ define the function
$f_x$ by
\[ f_x(t) =
\begin{cases}
f(x-t) \text{ if } 0 \leq t \leq x \\
0 \text{ if } x < t \leq 1 
\end{cases}
\]
\begin{align*}
|(V_f)\rho(x)| &= |\int_0^x f(x-t)\rho(t)\, dt| \\
	&\leq \int_0^x |f(x-t)|\, |\rho(t)|\, dt \\
	&\leq \int_0^1 (\sqrt{|f_x(t)|})\, (\sqrt{|f_x(t)|}|\rho(t)|) \, dt \\
	&\leq (\int_0^1 |f_x(t)|\, dt)^{\frac{1}{2}} \, (\int_0^1 |f_x(t)|\, |\rho(t)|^2 \, dt)^{\frac{1}{2}} \\
	&\leq ||f_x||_1^{\frac{1}{2}}\, (\int_0^1 |f_x(t)|\, |\rho(t)|^2 \, dt)^{\frac{1}{2}} \\
	&\leq ||f||_1^{\frac{1}{2}} \, (\int_0^1 |f_x(t)|\, |\rho(t)|^2 \, dt)^{\frac{1}{2}} \\
\end{align*}

Thus
\begin{align*}
 \int_0^1 |(V_f)\rho(x)|^2\, dx &\leq ||f||_1  \int_0^1 \int_0^1 |f_x(t)|\, |\rho(t)|^2\, dt\, dx \\
	&\leq  ||f||_1  \int_0^1 (\int_0^1 |f_x(t)|\, dx)|\rho(t)|^2\,  dt \\
	&\leq ||f||_1 ||f||_1 |||\rho|^2||_1 \\
	&\leq (||f||_1 \, ||\rho||_2)^2 
\end{align*}

Thus, $(V_f)\rho \in L^2[0, 1]$ and $||(V_f)\rho||_2 \leq ||f||_1\, ||\rho||_2.$ Thus $||V_f|| \leq ||f||_1.$

To prove (2), note that by definition,
\[ ||V_f|| = \sup \{ |<(V_f) \rho, \xi>| \ : ||\rho||_2 \leq 1, \  ||\xi||_2 \leq 1 \}. \]

For $f \in L^1[0, 1],$ define 
\[ \rho(t) = \ol{\text{sgn}(f(1-t))} :=
\begin{cases}
\frac{\ol{f(1-t)}}{|f(1-t)|} \text{ if } f(1-t) \neq 0 \\
0 \text{ otherwise}
\end{cases} \]
and let $F(x) = \int_0^x f(x-t) \rho(t) \, dt.$  Then $F$ is continuous, and $F(1) = ||f||_1.$  Now let $\xi(x) = \text{sgn}(F(x)).$
Since $||\rho||_2, \ ||\xi||_2 \leq 1,$ we have

\[ ||V_f|| \geq | <(V_f)\rho, \xi>| = \int_0^1 |F(x)|\, dx > 0. \]

For (3),  if the sequence of integrals is not bounded, then for all $ x_0 \leq x \leq 1,$ the values of $T(1)(x)$ are either $\pm \infty$ or undefined. But that
contradicts that $T(1) \in L^2[0, 1].$ 

For (4),  if the space $S_{x_0}$ is complete in the operator norm $f \mapsto ||V_f||,$ then since by part (1) $||V_f|| \leq ||f||_1,$ it follows from a
Corollary of the Open Mapping Theorem that the two norms are equivalent. Suppose that $S_{x_0}$ is not complete in the operator norm. Then there is
an element $T \in \sA_V, \ ||T|| = 1,$ and a sequence $\{f_n\} \subset S_{x_0}$ so that $\{V_{f_n}\}$ converges to $T$ with $\int_0^{x_0} |f_n|$ unbounded. But that
contradicts (3).

For (5), let $ 0 < x_0 < 1 $ and define the projection $P_{x_0}$ on $ L^2[0, 1]$ by
\[ P_{x_0}\rho(t) =
\begin{cases}
\rho(t) \text{ if } 0 \leq t \leq x_0 \\
0 \text{ otherwise }
\end{cases}
\]
For $f \in L^1[0, 1],$ let $f_{x_0} \in S_{x_0} $ be defined by 
$f_{x_0}(t) = 
\begin{cases}
f(t) \text{ if } 0 \leq t \leq x_0 \\
0 \text{ otherwise}
\end{cases} $
Observe that 
\[  P_{x_0} V_f P_{x_0} = P_{x_0} V_{f_{x_0}} P_{x_0} \]
Also note that $ f \in S_{x_0} \mapsto ||P_{x_0} V_f P_{x_0}||$ is a norm on $S_{x_0},$ weaker than the operator norm. 
That it is a norm follows from (2), with the interval $[0, 1]$ replaced by $[0, x_0].$

Let $T \in \sA_V$ and $\{f_n \}$ a sequence of functions in $L^1[0, 1]$ such that $\{V_{f_n}\}$ converges to $T.$ Then 
\[ \{P_{x_0} V_{f_n} P_{x_0} = P_{x_0} V_{f_{n, x_0}} P_{x_0} \} \text{ converges to }  P_{x_0}TP_{x_0} \]
In other words, the sequence $\{V_{f_n, x_0}\}$ converges in the norm defined in the previous paragraph. Since by (3) the integrals $\int_0^{x_0} |f_n|$ are bounded, the
limit of the sequence $\{f_{n, x_0} \}$  with respect to this norm belongs to $S_{x_0}. $
Thus, the limit of $ \{ V_{f_n, x_0} \}$ has the form $V_{g}$ for some $g \in S_{x_0}.$

Now if $ x_0 < y < 1,$ then $\{f_{n, y}\} $ converges with respect to the norm $ q \in S_y \mapsto ||P_y V_q P_y||$ to a function $h.$ In other words,
$ \{V_{f_{n, y}} \}$ converges in this norm to an operator $V_h$ for some $h \in S_y,$ and furthermore, the restriction of $h$ to the interval $[0, x_0]$
coincides with $g.$ Thus, if $x_0 < x_1 < x_2 < \cdots < 1$ with $\lim_n x_n = 1,$ then we obtain a sequence of functions $f_{x_n}$ such that
$f_{x_{n+1}}$ restricted to $[0, x_n]$ equals $f_{x_n}$ on that interval. Thus we obtain a function $f,$ measurable on $[0, 1],$ whose restriction
to $[0, x_n]$ is equals the restriction of $f_{x_n}$ to $[0, x_n].$ Finally, the fact that $f$ is integrable over compact subsets of $[0, 1)$ follows from
the fact that all of the $f_{x_n}$ are integrable.

To verify (6), let $\{p_n\}, \ \{q_n\}$ be sequences of polynomials such that $\{V_{p_n}\}, \{\ V_{q_n}\}$ converge to $V_f, \ V_g$ respectively.
Since $V_{p_n + q_n} = V_{p_n} + V_{q_n},$ and $V_{\al p_n} = \al V_{p_n},$ the conclusion follows.

\end{proof}

In \cite{LiRe98} G. Little and J. B. Reade prove an asympotic estimate for the norm of powers of the Volterra operator $V:$
\[ \lim_n n! ||V^n|| = \frac{1}{2} .\]

This result can be interpreted as an asymptotic estimate of the ratio of the operator norm to the $L^1$-norm on the set of functions $f_n(x) = x^n.$
Indeed, since $n!V^{n+1} = V_{f_n},$ we have
\[ \frac{||V_{f_n}||}{||f_n||_1} = \frac{n!||V^{n+1}||}{\frac{1}{n+1}} = (n+1)! ||V^{n+1}|| \to \frac{1}{2} \]
as $n \to \infty.$ Since the set of functions $\{f_n: \ n = 0, 1, \dots \}$ spans a dense subspace of the operator algebra $\sA_V,$ this result
seems to suggest that the two norms may be equivalent.

However, it turns out that the two norms are not equivalent, as the following example shows.

\begin{example} \label{e:notell1}

Let $\sS$ be the space of Lebesgue measurable functions $f$ on $[0, 1]$  such that, for every $ 0 < x < 1, \  f|_{[0, x]} \in L^2[0, x].$ 
For $ f \in \sS,$ define 
\[\rho_x(t) =
\begin{cases}
\frac{1}{c(x)}\ol{f(x - t)}, \text{ if } c(x) \neq 0 \text{ and } t \leq x \\
0, \text{ otherwise } 
\end{cases} \]
where $c(x) = [\int_0^x |f(x-t)|^2 \, dt]^{\frac{1}{2}} .$
Then, clearly, for any function $\rho \in L^2[0, 1]$ of unit norm, $ |\int_0^x f(x-t)\rho(t)\, dt | \leq \int_0^x f(x-t)\rho_x(t)\, dt.$

Thus, if $G(x) = (\int_0^x f(x-t)\rho_x(t) \, dt)^2,$ and if $f$ is not zero a.e. in the interval $[0, x],$ we have

\begin{align*}
G(x) &= \frac{1}{c(x)^2} (\int_0^x |f(x-t)|^2 \, dt)^2 \\
	&= \int_0^x |f(x-t)|^2 \, dt
\end{align*}
Hence $||V_f|| \leq [\int_0^1 G(x) \, dx]^{\frac{1}{2}}.$

Let $\sS_1 = \{ f \in \sS: \int_0^1 \int_0^x |f|^2 < \infty \},$ and for 

\[f \in \sS_1, \text{ set } ||f||_{\sharp} = [\int_0^1 \int_0^x |f|^2]^{\frac{1}{2}}. \]

If the kernel $k_f$ is defined by
\[ k_f(x, t) =
\begin{cases}
f(x-t) \text{ if } t \leq x \\
0 \text{ if } t > x
\end{cases}
\]
then the condition $ f \in \sS_1$ is equivalent to $k_f \in L^2([0, 1] \times [0, 1])$ and in that case $||f||_{\sharp} = ||k_f||_2.$  Thus $||f||_{\sharp}$ is the
Hilbert-Schmidt norm of the operator $V_f.$

To show that $\sA_V$ properly contains $L^1[0, 1]$ it suffices to exhibit a function $f$ with $||f||_{\sharp} < \infty$ but $f \notin L^1[0, 1].$

Now let $f = \sum_{n=1}^{\infty} \frac{2^n}{n}\, \chi_{[1- 2^{-(n-1)}, 1 - 2^{-n})}.$ Then $ \int_0^1 f = \sum_{n=1}^{\infty} \frac{1}{n} $ diverges, so $f \notin L^1[0, 1].$
However, $||f||_{\sharp}^2 = \int_0^1 G$ is finite. To see this, view the area under the graph of $G$ as divided into horizontal strips. The portion of the area between $G(1 - 2^{-(n-1)})$ and
$G(1- 2^{-n})$ is $\frac{3}{2}\cdot \frac{1}{n^2}.$ Thus,

\[ \int_0^1 G = \sum_{n=1}^{\infty} \frac{3}{2}\cdot \frac{1}{n^2} = \pi^2/4.\]

It follows that the operator norm $ ||V_f||$ is finite, and in fact \mbox{$||V_f|| \leq \pi/2.$} If $f_n := f\chi_{[0, 1-2^{-n})},$ then $f_n \in L^1[0, 1]$ and a calculation
similar to the above shows that $||V_f - V_{f_n}|| \to 0$ as $n \to \infty.$ Since $V_{f_n} \in \sA_V$ and $\sA_V$ is by definition complete, $V_f \in \sA_V.$

Thus, $ V_f \in \sA_V,$ but $f \notin L^1[0, 1].$
\end{example}

\begin{remark} \label{r: HS}
There is no simple relationship between the Hilbert Schmidt norm of $V_f$ and $||f||_1.$
We have just seen that the Hilbert-Schmidt norm of $V_f$ can be finite and $f$ not in $L^1[0, 1].$
On the other hand,  if $f \in L^1[0, 1]$ is such that, for some $ 0 < x < 1, \ f|_{[0, x]} \notin L^2[0, x],$
then the $L^1$ norm of $f$ is finite but the Hilbert Schmidt norm of $V_f$ is not. 

\end{remark}

\begin{example} \label{e: not in Volterra}
If $f$ is measurable on $[0, 1],$ the condition $k_f \in L^1([0, 1]^2)$ does not imply $V_f \in \sA_V.$
Let  $ f(x) = \frac{1}{(1 - x)^{\frac{3}{2}}}.$  Then

\[||k_f||_1 = \int_0^1 (\int_0^x f(t) \, dt)\, dx < \infty \]

To show $V_f$ is unbounded, it is enough to find $\rho \in L^2[0, 1]$ such that $(V_f) \rho \notin L^2[0, 1].$
Take $\rho$ to be the constant $1.$ Then
\[ ||(V_f) \rho||_2^2 = \int_0^1 | \int_0^x f(t)\, dt|^2 \, dx \text{ diverges.} \]

This shows that that the conditions of Theorem~\ref{t: Vf in L2} part (5) on a measurable function $f$ to satisfy $V_f \in \sA_V$ are necessary but not sufficient.
\end{example}


\subsection{Operator Algebraic properties of $\sA_V$}

We turn now from a discussion of the norm of operators in $\sA_V$ to its properties as an algebra. Since the Volterra operator is quasinilpotent, the algebra $\sA_V$ 
is a commutative radical operator algebra. From the example of the weighted shift, we know that radical operator algebras $\sA_T$ can admit a gauge action. 
Does the same hold for $\sA_V ?$

\begin{proposition} \label{p: Volterra.gauge}
$\sA_V$ does not admit a gauge action. 
\end{proposition}

\begin{proof}

We will make use of the formula for $V^{n+1}$ (cf equation \ref{eq: Vn}).

Now by the Muntz-Szasz Theorem, the function $g(x) = x$ can be uniformly approximated in $[0, 1]$ by polynomials in $\{x^2, x^3, \dots\}.$ 
 Thus given $\ep > 0,$ we can find a polynomial $p(x) = a_2x^2 + a_3x^3 + \cdots + a_n x^n$ satisfying
$||g - p||_{\infty} =\sup_{0 \leq x \leq 1} |x - p(x)| < \ep.$

Let $\rho \in L^2[0, 1]$ with $||\rho||_2 = 1.$ Then
\begin{align*}
||V^2(\rho) - \sum_{j=2}^n j! a_j V^{j+1}(\rho)||_2 &=  ||\int_0^x [(x-t) - p(x-t)] \rho(t) \, dt ||_2  \\
	&\leq ||\int_0^x |(x-t) - p(x-t)| |\rho(t)| \, dt ||_2 \\
	&\leq || \int_0^1 \ep |\rho(t)|\, dt||_2 \\ 
	&\leq \ep
\end{align*}

Since this holds for all $\rho \in L^2[0, 1]$ of norm $1,$ it follows that
\[ ||V^2 - \sum_{j=2}^n  j! a_j V^{j+1}|| \leq \ep .\]

Now we assume that the operator algebra $\sA_V$ admits a gauge action, $\ga.$  Since by definition the Fourier coefficient $\hat{V^2}(2) = 1,$ and
$\widehat{(\sum_{j=2}^n j! a_j V^{j+1} )}(2) = 0,$ and since the operation $ S \in \sA_V  \mapsto \hat{S}(2) = \int_{\bbT} \ga_{\la}(S) \la^{-2}\, d|\la|$
is norm-decreasing, it follows that
\[ ||V^2|| = ||\int_{\bbT} (\ga_{\la}(V^2 - \sum_{j=2}^n j! a_j V^{j+1} ) \la^{-2} \, d|\la| \, ||\leq \ep \]
which is absurd, since $\ep > 0 $ was arbitrary.

Thus $\sA_V$ does not admit a gauge action.

\end{proof}

If $T$ is a quasinilpotent weighted shift, then the fact that the lattice $\Lat{T}$ is discrete and $\Lat{V}$ is continuous tells us that
the two operators are not unitarily equivalent. But to show that the operator algebras $\sA_T, \ \sA_V$ are not isomorphic requires another argument.

\begin{corollary} \label{c: not isomorphic}
Let $T$ be a quasinilpotent weighted shift. Then the radical operator algebras $\sA_T, \ \sA_V$ are not completely isometrically isomorphic.
\end{corollary}

\begin{proof} Since $\sA_T$ admits a gauge action, and $\sA_V$ does not, they cannot be completely isometrically isomorphic.
\end{proof}

Define a `convolution' on elements of $\sA_V$ as follows: if $f, \ g $ are measurable functions on $[0, 1]$ such that $V_f, \ V_g \in \sA_V,$ set
\begin{equation} \label{eq: convolution}
 f*g(x) = \int_0^x f(x-t) g(t) \, dt 
\end{equation}
First observe that since the restrictions of $f, g$ to the interval $[0, x]$ are integrable if $ x < 1,$ it follows that $f*g$ is well defined. Furthermore, 
a calculation shows that, for $\rho \in L^2[0, 1],$ 
\begin{equation} \label{eq: convolution2}
   V_{f*g}(\rho) = V_f(V_g(\rho))  \text{ and hence } V_{f*g} = V_f V_g .
\end{equation}

It is well known that the closed invariant subspaces of the Volterra operator $V$ have the form $\{ \xi \in L^2[0, 1]: \xi(t) = 0 \text{ for } 0 \leq t \leq x_0 \},$
for $ 0 < x_0 < 1.$ (\cite{KDav1988}, Theorem 5.5. Note that their notation  differs from our: their $V$ is $V^*$ here.) 
The same holds for the operator algebra, $\sA_V.$  This yields some information about the closed ideals of $\sA_V.$

\begin{proposition} \label{p: ideals of AV}
Let $ 0 < x_0 < 1.$ The subspace $ \sI_{x_0} = \{V_f: \  f(t) = 0, \text{ for } 0 \leq t \leq x_0\}$ is a closed ideal of $\sA_V.$
\end{proposition}

\begin{proof}
Let $V_g \in \sI_{x_0}$ and $V_f \in \sA_V.$  Since $V_f V_g = V_{f*g},$ and since we have $f*g(x) = \int_0^x f(x-t) g(t),$ it is clear that if $x \leq x_0$
then $f*g(x) = 0.$ Thus, $\sI_{x_0}$ is an ideal.

The proof that $\sI_{x_0}$ is a closed ideal is similar to the proof of statement (2) of Theorem~\ref{t: Vf in L2}.
Suppose $V_f \in \sA_V,\ V_f \notin \sI_{x_0}.$ This implies that the restriction of $f$ to the interval $[0, x_0]$ is not zero.

Let $F(x) = \int_0^x f(x-t)\rho(t) \, dt,$ where
\[ \rho(t) =
\begin{cases}
sgn(f(x_0 -t)), \text{ if } t \leq x_0 \\
0, \text{ if } t > x_0
\end{cases} \]

Then $F$ is continuous, and $F(x_0) = \int_0^{x_0} |f(x_0 - t)| \, dt > 0.$
Let
\[ \xi(t) =
\begin{cases}
sgn(F(x)), \text{ if } x \leq x_0 \\
0, \text{ if } x > x_0
\end{cases} \]

Now let $V_g \in \sI_{x_0}$ be arbitrary. Then
\[ <(V_f - V_g)(\rho), \xi> = \int_0^{x_0} |F(x)|\, dx \text{ is a positive constant.} \]
It follows that no sequence $\{ V_{g_n} \} \subset \sI_{x_0}$ converges to $V_f.$ Hence $\sI_{x_0}$ is a closed ideal.

\end{proof}

\begin{remark} \label{r: ideals and subspaces}
In Proposition~\ref{p: ideals and subspaces} we established a one-to-one correspondence between closed ideals of the operator algebra $\sA_T$
of the weighted shift $T,$ and invariant subspaces. If such a relationship were to exist  for the Volterra operator algebra, then we would
have a complete description of the closed ideals of $\sA_V.$

Note that if an operator algebra $\sA$ is completely isometrically represented on a Hilbert space $H$ with cyclic vector $\xi_0,$ then given an invariant subpace $H_1$
there is a closed ideal $\sI$ by $ \sI = \{ a \in \sA: a\xi_0 \in H_1\}.$ In the other direction, there is no assurance that if $\sI $ is a closed ideal of $\sA,$
that the subspace $\sI \cdot \xi_0$ is closed in $H.$

\end{remark}

\begin{lemma} \label{l: support}
Let $ \al, \be \in (0, 1)$ and suppose $f, g$ are measurable functions on $[0, 1]$ such that $V_f, V_g \in \sA_V,$ and
\begin{enumerate}
\item If $ f$  is supported on $ [\al, 1]$  and $ g $  is supported on  $ [\be, 1],$
then $f*g$ is supported on $[(\al + \be), 1]$ if $\al + \be < 1,$ and $f*g = 0$ if $\al + \be \geq 1.$

\item In particular, if $\be = 1 - \al,$ then $f*g = 0.$
\end{enumerate}
\end{lemma}

\begin{proof}
The proof is a straightforward computation, making use of the convolution formula, equation~\ref{eq: convolution}.

\end{proof}

\begin{corollary} \label{c: nilpotent}
If $f$ is a measurable function on $[0, 1]$ such that $V_f \in \sA_V$ and $f$ is supported on $[\al, 1],$
then $(V_f)^n = 0$ if $n\al \geq 1.$
\end{corollary}

\begin{proof}
This follows from repeated application of Lemma~\ref{l: support}
\end{proof}

The following result appears as Proposition 2.13 of \cite{PeWo99}. It is proved there using methods entirely different from those in this paper.
\begin{corollary}
The nilpotent elements in the Volterra operator algebra $\sA_V$ are dense.
\end{corollary}

\begin{proof}
Let  $ V_f \in \sA_V.$ Since the elements of the form $V_g, \ g \in L^1[0, 1]$ are dense in $\sA_V,$ given $\ep > 0,$ let $g \in L^1[0, 1]$  satisfy $||V_f - V_g|| < \ep.$
There is $\de > 0$ such that if $ E \subset [0, 1]$ is measurable with $\mu(E) \leq \de,$ then $\int_E |g| \leq \ep.$

Define $h \in L^1[0, 1]$ by
\[ h(x) =
\begin{cases}
0, \text{ if } 0 \leq x < \de \\
g(x), \text{ if } \de  \leq x \leq 1
\end{cases}\]
Then
\begin{align*}
||V_f - V_h || &\leq ||V_f - V_g|| + ||V_g - V_h|| \\
	&< \ep + ||g - h||_1 \\
	&< 2\ep
\end{align*}

Observe that the second inequality above makes use of part (1) of Theorem~\ref{t: Vf in L2}, since
\[ ||V_g - V_h|| = ||V_{g-h}|| \leq ||g - h||_1 \]

Let $n \in \bbN$ satisfy $n \de \geq 1.$ It follows from Corollary\ref{c: nilpotent} that $V_h$ is nilpotent in $\sA_V.$
\end{proof}

\begin{remark} \label{r: Titschmarsh}
$L^1[0, 1]$ is a Banach algebra under the convolution $ f*g(x) = \int_0^x  f(x-t) g(t) \, dt.$   Hence this Banach algebra is dense in $\sA_V,$ and the convolution
is the restriction of that in $\sA_V$ to $L^1.$ 

One version of the Titschmarsh convolution is: if $f, g \in L^1[0, 1]$ and
$ f*g = 0 $ then there is an $\al  \in [0, 1]$ such that $\text{supp} f \subset [\al, 1]$ and $\text{supp} g \subset [1-\al, 1].$ (\cite{KDav1988} Problem 5.4 and
\cite{Lax} Theorem 10, Sec. 38.3)

The next result shows that the Titschmarsh convolution theorem also holds in the larger algebra $\sA_V.$
In particular, the converse of statement (2) of Lemma~\ref{l: support} holds.

\end{remark}

\begin{corollary} \label{c: Titschmarsh}
Let $f, g$ be real-valued measurable functions on $[0, 1]$ such that $V_f, \ V_g \in \sA_V.$ Then
 $V_{f*g} = 0$ if and only if there exists $\al \in [0, 1]$ such that $\text{supp} f \subset [\al, 1]$ and $\text{supp} g \subset [1 - \al, 1].$
\end{corollary}

\begin{proof}
One direction follows from Lemma~\ref{l: support}

If $V_{f*g} = V_f V_g = 0, $ then $V(V_f V_g) = V_F V_g = 0$ where $F(x) = \int_0^x f(t) \, dt.$ 
 It follows that $V_F V_g \xi = 0$ for all $\xi \in L^2[0, 1].$ Thus $ \int_0^x F(x-t) (V_g\xi)(t) \, dt = 0$ for a.a. $ x \in [0, 1].$

As $F = V_f(1)$ and $ V_g\xi $ are both in $ L^2[0, 1] \subset L^1[0, 1],$ it follows from the classical Titschmarch theorem that there exists $\al \in [0, 1]$ such that
$F = 0 $ in $[0, \al]$ and $ \int_0^x g(x-t) \xi(t) \, dt = 0$ for  $x \in [0, 1- \al].$  Thus $f = 0 $ in $[0, \al].$

Suppose $ g \neq 0$ in $[0, 1 - \al].$ Thus for some $x_0 \in [0, 1-\al], \ \int_0^{x_0} |g|\, dt \neq 0.$ Define
$\xi(t) = sgn(g(x_0-t)).$ Then $\int_0^x g(x-t) \xi(t) \, dt $ is a nonzero continuous function in the interval $[0, 1 - \al],$ a contradiction.

\end{proof}

The author would like to thank Chris Phillips for his comments on an earlier version of this paper.



\begin{thebibliography}{99}
\bibitem{BRS1990} Blecher, D. P.; Ruan, Z-J, and Sinclair, A. M., \textit{A characterization of operator algebras}, J. Funct. Anal. 89, no. 1, 188–201, 1990.

\bibitem{KDav1988} Davidson,  K. R., Nest Algebras: triangular forms for operators on Hilbert space, Longman Scientific, 1988.

\bibitem{KDav96} Davidson, K. R.,  C$^*$-Algebras by Example, Fields Institute Monographs, American Mathematical Society, 1996.

\bibitem{DaKa2011} Davidson, K. R. and  Katsoulis, E. G., \textit{Dilation Theory, Commutant Lifting and Semicrossed Products}, Documenta Math., 2011.

\bibitem{HPP2005} Hopenwasser, A., Peters, J. R., and Power, S. C., \textit{Subalgebras of Graph C$^*$-algebras}, New York J. Math 11, 351 - 386, 2005.

\bibitem{KHof62} Hoffman, K., Banach Spaces of Analytic Functions, Prentice-Hall, 1962.

\bibitem{KatRam19} Katsoulis, E. G. and  Ramsey, C., Crossed Products of Operator Algebras, AMS Memoir no. 1240, 2019.

\bibitem{Lax} Lax, P., Functional Analysis, Wiley, 2002.

\bibitem{LiRe98} Little, G. and Reade, J. B., \textit{Estimates for the Norm of the nth Indenite Integral,}
Bull. London Math. Soc. \textbf{30}, no. 5, 539 - 542, 1998.

\bibitem{VPaulsen86} Paulsen, V.,  Completely bounded maps and dilations, Pitman Research Notes in Mathematics Series, 1986.

\bibitem{PeWo99} Peters, J. R.; Wogen, W. R. Commutative radical operator algebras. J. Operator Theory 42 (1999), no. 2, 405--424. 

\bibitem{RW22} Ransford, T. and  Walsh, N., \textit{Norms of Polynomials of the Volterra Operator}, 
J. Math. Anal. Appl. 517, no. 2, Paper No. 126626, 15 pp., 2023.

\bibitem{Rid70} Ridge, William C., \textit{Approximate point spectrum of a weighted shift}, Trans. Amer. Math. Soc.147, 349--356, 1970

\bibitem{Sh74} Shields, A. L., \textit{Weighted shift Operators and Analytic Function Theory}, Topics in Operator Theory, Math. Surveys Monogr., vol. 13,
Amer. Math. Soc., Providence, R. I., 49--138, 1974.


\end{thebibliography}
\end{document}